\documentclass[12pt]{article}
\usepackage{amsthm,amsmath,amssymb,color}
\usepackage{array}
\usepackage{fullpage}
\usepackage{graphicx}
\usepackage{ytableau}
\usepackage{tikz}
\usepackage{kotex}
\usepackage[shortlabels]{enumitem}
\usepackage{mathtools}
\usepackage{hyperref}
\usepackage{float}
\usepackage[numbers,sort&compress]{natbib}

\usepackage{color}
\usepackage[normalem]{ulem}

\newtheorem{thm}{Theorem}[section]

\newtheorem{lem}[thm]{Lemma}
\newtheorem{pro}[thm]{Proposition}
\newtheorem{Obs}[thm]{Observation}

\newtheorem{ex}[thm]{Example}

\newtheorem*{prol}{Problem}
\theoremstyle{remark}

\date{}

\begin{document}

\title{The first Steklov eigenvalue bound for graphs of positive genus}
\author{Lixiang Chen, Yongtang Shi and Liwen Zhang \\
\footnotesize Center for Combinatorics and LPMC, Nankai University, Tianjin, 300071, PR China\\
\footnotesize  clxmath@163.com, shi@nankai.edu.cn and levenzhang512@163.com }

\maketitle

\begin{abstract} 
Let $G$ be a graph of genus $g$ with boundary $\delta\Omega$. For $g=0$, Lin and Zhao [J. Lond. Math. Soc. 112 (2025), Paper No. e70238] proved an upper bound for the first (non-trivial) Steklov eigenvalue of $(G, \delta\Omega )$, and they posed the problem of determining a corresponding bound for graphs of genus $g>0$. 
In this paper, we prove an $O\left(\frac{g}{|\delta \Omega|}\right)$ bound for a bounded-degree graph of positive genus $g$. Our result can be regarded as a discrete analogue of Kokarev's bound [Adv. Math. 258 (2014), 191-239], up to a constant factor.
\\

	

\end{abstract}

\section{Introduction}
The Steklov eigenvalue problem, originating from Steklov's century-old work on liquid sloshing \cite{Kuznetsov2014,Stekloff1902}, is an active field connecting spectral theory, geometry, and mathematical physics.
Let $M$ be a compact smooth orientable Riemannian manifold of dimension $d \geq 2$ with a
smooth boundary $\partial M$. The Dirichlet-to-Neumann operator $\mathcal{D}: C^{\infty}(\partial M) \to C^{\infty}(\partial M) $
is defined by $\mathcal{D}f=\partial_{\nu}\hat{f}$, where $\partial_{\nu}$ is the outward normal derivative
along $\partial M$, and where the function $\hat{f} \in C^{\infty}( M)$ is the unique harmonic extension of $f$ to the interior of $M$.
The eigenvalues of $\mathcal{D}$, known as the Steklov eigenvalues of $M$, are discrete and can be ordered as
$0=\lambda_1\leq\lambda_2 \leq \lambda_3 \leq \cdots \to \infty$.
Here, $\lambda_2$ is called the first (non-trivial) Steklov eigenvalue.

The initial motivation for the geometric study of Steklov eigenvalues comes from Weinstock's 1954 work \cite{Weinstock1954}, where he demonstrated that, for any simply connected planar domain of a given perimeter, $\lambda_2$ attains its maximum if and only if the domain is a disk.
For bounded Lipschitz domains of fixed volume in Euclidean space, Brock  \cite{Brock2001} proved that $\lambda_2$ is maximized by the ball.
Colbois, Girouard and Raveendran \cite{Colbois2018} provided upper bounds for $\lambda_k$ of domains in Euclidean space. For a detailed account of recent advances on the Steklov eigenvalue problem, see the survey  \cite{ColboisGirouardGordonSher2023}.

As has been done for the Laplace operator, one can define a discrete Steklov eigenvalue problem \cite{Hassannezhad2020, Hua2017}, which is defined on graphs with boundaries. Several results on the bounds for the first Steklov eigenvalue of graphs are obtained.   
The lower bounds are established by Perrin \cite{Perrin2019} and Shi-Yu \cite{Shi2022,ShiYu2025}.
The upper bounds are studied in various settings, including: for finite subgraphs of integer lattices by Han-Hua \cite{Han2023}, for subgraphs of Cayley graphs of discrete groups with polynomial growth by Perrin \cite{Perrin2021}, for trees by He-Hua \cite{He2022} and Lin-Zhao \cite{Linbull2025,Lintree2025}, and for planar graphs and block graphs by Lin-Zhao \cite{lin2025steklov}. 

All graphs in this paper are finite, simple, and undirected graphs with a non-empty boundary. 
For a graph $G=(V,E)$ with boundary $\delta \Omega \subset V$,
denote by $\mathbb{R}^{V}$ the real function space defined on $V$, which is the $|V|$-dimensional Euclidean space.  
For $f\in \mathbb{R}^V$, the Laplace operator $\Delta:\mathbb{R}^V \to \mathbb{R}^V$ is defined by 
\begin{align*}
    \Delta f:~&V \to \mathbb{R}\\
    &x \mapsto (\Delta f)(x)=\sum_{\{x,y\}\in E}\left(f(x)-f(y)\right),
\end{align*}
and the discrete outward normal derivative $\partial_{\nu}:\mathbb{R}^V \to \mathbb{R}^{\delta \Omega}$ is defined by 
\begin{align*}
   \partial_{\nu}f:~& \delta\Omega \to \mathbb{R}\\
    &x \mapsto (\partial_{\nu}f)(x)=\sum_{\{x,y\}\in E}\left(f(x)-f(y)\right).
\end{align*}
The Steklov eigenvalue problem on $(G,\delta\Omega)$ aims to find real numbers $\lambda \in \mathbb{R}$ and non-trivial functions $f \in \mathbb{R}^V$ such that 
\begin{align*}
    \begin{cases}
       (\Delta f) (x)=0,  &\mbox{if $x\in V\setminus\delta\Omega $},\\
        (\partial_{\nu}f)(x)=\lambda f(x), & \mbox{if $x\in \delta \Omega $}.
    \end{cases}
\end{align*}
 Such a $\lambda$ is called a Steklov eigenvalue of the graph $G$ with  boundary  $\delta \Omega$, and the function $f$ is called the Steklov eigenfunction associated to $\lambda$. 
The Steklov spectrum forms a sequence as follows
\begin{align*}
0=\lambda_1(G,\delta\Omega)\leq  \lambda_2(G,\delta\Omega)  \leq \cdots\leq \lambda_{|\delta\Omega|}(G,\delta\Omega). 
\end{align*}
In fact, $\lambda_2(G,\delta\Omega)>0$ if and only if $G$ is connected. 
The Steklov eigenvalues can be characterized by the Rayleigh quotient
\begin{align}\label{eq:Rq1}
  \lambda_k(G,\delta \Omega)
  &= \min\limits_{\substack{W\subset \mathbb{R}^{V}, \dim W = k}}
     \max\limits_{f \in W} R(f) \notag \\
  &= \min\limits_{\substack{W\subset \mathbb{R}^{V}, \dim W = k-1, \\ W\perp \mathbf{1}_{\delta \Omega}}}
     \max\limits_{f \in W} R(f),
\end{align}
where
\begin{equation*}
  R(f)=\frac{\sum_{\{x,y\} \in E}(f(x)-f(y))^2}{\sum_{x\in \delta \Omega}f^2(x)}
\end{equation*}
and $\mathbf{1}_{\delta \Omega}$ denotes the characteristic function on $\delta \Omega$. Analogous to normalized Laplacian eigenvalues discussed in Chung's book \cite[Chapter 1]{Chung1997}, when $k=2$, \eqref{eq:Rq1} has the following formulations
\begin{align}
  \lambda_2(G,\delta \Omega)&=\min_{f\perp \mathbf{1}_{\delta \Omega}}\frac{\sum_{\{x,y\} \in E}(f(x)-f(y))^2}{\sum_{x\in \delta \Omega}f^2(x)}  \label{eq:orig}  \\
  &=|\delta \Omega|\min_f\frac{\sum_{\{x,y\} \in E}(f(x)-f(y))^2}{\sum_{\{x,y\}\subset \delta \Omega}(f(x)-f(y))^2}. \label{eq:Ray}
\end{align}


The \textit{genus} $g$ of a graph $G$ is the minimal integer such that $G$ can be embedded on a surface of genus $g$ with none of its edges crossing. For a planar graph $G$ (of genus $g=0$) with boundary $\delta\Omega$, Lin and Zhao \cite[Theorem 1.1]{lin2025steklov} gave an insightful bound $\lambda_2(G,\delta\Omega) \leq \frac{8D}{|\delta\Omega|}$,
where $D$ is the maximum degree of $G$. Furthermore, they posed the following problem \cite[Problem 3.2]{lin2025steklov} regarding graphs of genus $g$. 

\begin{prol}\cite[Problem 3.2]{lin2025steklov}
    What is the upper bound of Steklov eigenvalues for genus $g$ graphs with boundary $\delta\Omega$ when the degree is bounded above by $D$?
\end{prol}

Using metrical deformation via flows, Lin, Liu, You, and Zhao \cite{Lin2024} proved a bound for graphs of positive genus $g$.
\begin{thm}\label{thm-lin}\cite[Theorem 1.2]{Lin2024}
Let $G$ be a graph of positive genus $g$ with boundary $\delta\Omega$ such that
the vertex degree is bounded by $D$. If $|\delta\Omega| \geq  \max \{3\sqrt{g},|V|^{\frac{1}{4}+\epsilon},9\}$ for some $\epsilon>0$, then
\begin{equation*}
\lambda_2(G,\delta\Omega) \leq O\left(\frac{D g^3}{|\delta\Omega|}\right).
\end{equation*}
\end{thm}

A graph $G$ is said to be \textit{bounded-degree} if its maximum degree $D$ is bounded by an absolute constant.
For a bounded-degree graph, we give the following result, in which the dependence on $g$ is $O(g)$.
\begin{thm}\label{thm:main}
Let $G$ be a bounded-degree graph of positive genus $g$ with boundary $\delta\Omega$. 
Then
\begin{equation*}
  \lambda_2(G,\delta \Omega)\leq O\left(\frac{g}{|\delta \Omega|}\right).
\end{equation*}
\end{thm}
For an orientable compact surface $M$ of genus $g$ with boundary $\partial M$, Kokarev \cite[Theorem $A_1$]{Kokarev2014} showed that $\lambda_2(M,\partial M)\leq \frac{8\pi(g+1)}{|\partial M|}$. Our result can be viewed as a discrete analogue of Kokarev's result, up to a constant factor.

\section{Proof idea}
Our proof is based on the framework of Spielman-Teng \cite{ST2007} for bounding the Fiedler value of planar graphs. 
They represented a planar graph as a circle packing on the unit sphere in $\mathbb{R}^3$, with vertices corresponding to circle centers and edges corresponding to tangencies between circles.
By analyzing the geometric distribution of circle centers, they derived an upper bound on the graph's Fiedler value.
In a recent work, Lin and Zhao \cite{lin2025steklov} applied this framework to bound the first Steklov eigenvalue of planar graphs (genus $g=0$).
For positive $g$, Kelner generalized the result of Spielman-Teng to graphs of genus $g$.
Following this line of work, we give a bound for the first Steklov eigenvalue of graphs of genus $g$.

The outline of our proof is as follows. First, we construct a univalent circle packing for a graph on a Riemann surface $S$ of genus $g$ by Theorem \ref{thm:CPT}.
Although it is very challenging to construct a discrete analytic map that sends this circle packing to the Riemann sphere $\widehat{\mathbb{C}}$, 
there exists a continuous one that is almost everywhere  $O(g)$-to-one, except for $O(g)$ points where this property fails (see Theorem \ref{thm:RR}). 
We shall actually pass to the continuous theory to prove our result.
To do this,  we refine the graph so that it admits a circle packing with arbitrarily small circles. Lemma \ref{lem:main}  shows that it suffices to prove Theorem \ref{thm:main} for such refined graphs. 
We then show that there is a continuous map that takes \emph{almost} all of the circles to \emph{almost circular} curves on the unit sphere in $\mathbb{R}^3$.
Finally, using the \emph{almost-circle packing} representation, we derive the desired upper bound on the Steklov eigenvalues.

In the remainder of this section, we  introduce several necessary notions and present auxiliary results for the proof.
For the sake of continuity, all proofs are deferred to Section \ref{sec:proof}.

\subsection{The map from a Riemann surface to the sphere}
This section provides geometric intuition by briefly recalling the basic notions of Riemann surfaces, circles on these surfaces, and the degrees and branch points of analytic maps between them.  For a more detailed exposition, see \cite{Forster1981}.

Recall that an \emph{$n$-dimensional manifold} is a structure that looks locally like a Euclidean space $\mathbb{R}^n$.
More formally, a manifold $M$ can be covered by open sets $S_i$, each endowed with a homeomorphism $\varphi_i:S_i\rightarrow B_n$, 
where $B_n$ is the (open) unit ball $\{|x|<1: x \in \mathbb{R}^n\}$. 
To avoid cusps and such, these maps are required to be mutually compatible. 
In particular, for any overlapping charts, the compositions 
$\varphi_j\circ\varphi_i^{-1}: \varphi_i(S_i\cap S_j)\rightarrow \varphi_j(S_i\cap S_j)$ are diffeomorphisms. 
The two-dimensional orientable manifolds are precisely surfaces of genus $g$, as we need.

An \emph{$n$-dimensional complex manifold} is the natural holomorphic generalization of the real manifold.
Similarly, it can be represented as a union of open sets $S_i$, each endowed with a homeomorphism $\varphi_i:S_i\rightarrow B_{\mathbb{C}^n}$, 
where $B_{\mathbb{C}^n}$ is the complex unit ball $\{|x|<1: x \in \mathbb{C}^n\}$. 
It is required that the compositions $\varphi_j\circ\varphi_i^{-1}$ be biholomorphic.
Thus, an $n$-dimensional complex manifold $M$ is a $2n$-dimensional real manifold equipped with an additional complex analytic structure.
This structure allows the use of many notions and results from standard  complex analysis. 
The basic idea is to define these concepts locally on each $S_i$ as before, and the compatibility conditions on overlaps ensure that they are well-defined globally. 
Let $M=(S_i^{M},\varphi_i^M)$ and $N=(S_i^{N},\varphi_i^N)$ be complex manifolds of the same dimension. A function $f:M \to N$ is called \emph{a holomorphic map} if its restriction to each map $f_{ij}:S_i^{M} \to S_j^{N}$ is holomorphic for all $i$ and $j$. This definition is well-defined on overlaps, since the compositions $\varphi_j^M \circ (\varphi_i^M)^{-1}$ and $\varphi_j^N \circ (\varphi_i^N)^{-1}$ are holomorphic.

A one-dimensional complex manifold is called a \emph{Riemann surface}. In this paper, all Riemann surfaces considered are compact. 
There is a natural orientation on the complex plane, which induces an orientation on the complex manifold. 
Hence, all complex manifolds (including Riemann surfaces) are orientable. 
Topologically, all compact Riemann surfaces are two-dimensional orientable real manifolds. 
However, the same underlying topological space can support a wide variety of distinct complex structures, which are imposed by the map $\varphi_i$.

The definition of a Riemann surface does not involve a metric structure.
Indeed, it is not required that the different $\varphi_i$ agree on distance measurement over overlapping regions.
Although a Riemann surface can be endowed with various metrics, there exists a canonical choice, known as the \emph{constant curvature metric}.  
It provides an intrinsic notion of distance on any Riemann surface,  which allows one to define a \emph{circle on the surface} as a contractible simple closed curve whose points all lie at a fixed distance from a given center.

The \emph{Riemann sphere} $\widehat{\mathbb{C}}$ is topologically a sphere. 
It can be viewed as the complex plane $\mathbb{C}$ together with an additional point, the point at infinity $\infty$, which makes the plane compact.
The correspondence can be visualized via the stereographic projection, which projects the sphere from its North Pole onto the plane.
More precisely, every point on the sphere except the North Pole corresponds to a point in the complex plane, while the North Pole corresponds to the point at infinity $\infty$.

Let $f: \mathbb{C} \to \mathbb{C}$ be a non-constant analytic in a neighborhood of the origin. 
Without loss of generality, set $f(0)=0$. 
By a result in single-variable complex analysis, the function $f$ can be expanded as a power series $f(z)=a_1z+a_2z^2+a_3z^3+\cdots$ in some neighborhood of the origin.
Let $a_k$ be the first nonzero coefficient in this expansion.  
If $k=1$, then $f(z)$ is an invertible analytic and hence a local isomorphism. 
In particular, it is a \emph{locally conformal map}, which preserves angles between intersecting curves and sends infinitesimal circles to infinitesimal circles. 
If $k >1$, then $f$ has \emph{a branch point of order $k$} at the origin. 
In this case, $f$ behaves like $f(z)=z^k$ up to a scale factor and higher-order infinitesimal terms. 
This function is $k$-to-one on a small neighborhood of the origin, excluding the origin itself. 
The preimages of the points in this small neighborhood thus trace out $k$ different ``sheets'' that all intersect at the origin. 
This confluence of sheets is the only type of singularity that can occur in an analytic map.

Let $f:M \to N$ be an analytic map between two Riemann surfaces. 
Since the local behavior discussed above is identical for Riemann surfaces, the map $f$ has a well-defined \emph{degree} $k$.
Except for finitely many points, each point in $N$ has exactly $k$ preimages under $f$, and $f$ has a nonzero derivative at all such points.
A branch point is a point  $p \in M$ where $f'(p)=0$, and $f(p)$ has fewer than $k$ preimages.

We invoke the following result of Kelner, which is a direct consequence of the Riemann-Roch theorem for maps between Riemann surfaces.

\begin{thm} \label{thm:RR} \cite[Theorem 4.4]{Kelner2006Spec}
Let $S$ be a Riemann surface of genus $g$. There exists an analytic map $f: S\rightarrow \widehat{\mathbb{C}}$ of degree $O(g)$ and with $O(g)$ branch points.
\end{thm}

\subsection{Circle packings on surfaces of arbitrary genus}
Let $G$ be a graph of genus $g$ embedded on a surface $S$ of the same genus without edge crossings, so that the embedding partitions $S$ into faces.
If every face is a triangle, then $G$ is said to be \emph{fully triangulated}, and the embedding is called a \emph{triangulation} of $S$.
If $G$ is not fully triangulated, additional edges can be added to make it.
By \eqref{eq:Rq1}, adding edges does not decrease $ \lambda_k(G,\delta \Omega)$,
and thus we may, without loss of generality, assume that $G$ is fully triangulated.

Let $S$ be a compact Riemann surface endowed with its constant curvature metric.
A \emph{circle packing} $\mathcal{P}=\{C_1,\dots,C_n\}$ on $S$ is a finite collection of possibly overlapping circles on $S$. 
The packing $\mathcal{P}$ is said to be \emph{univalent} if the interiors of the circles $C_i$ are pairwise disjoint.
Its \emph{associated graph} $A(\mathcal{P})$ is defined by assigning a vertex $v_i$ to each circle $C_i$ and joining $v_i$ to $v_j$ by an edge if and only if $C_i$ and $C_j$ are tangent.
Conversely, we say that $\mathcal{P}$ is a circle packing for the graph $A(\mathcal{P})$ on $S$.

We will use the following circle packing theorem, which can be regarded as a natural generalization of the Koebe-Andreev-Thurston theorem  \cite{andreev1970convex,Koebe1936,Thurston3M} for planar graphs.
The theorem was first proved in a restricted form by Beardon-Stephenson \cite{beardon1990uniformization}, and subsequently established in full generality by He-Schramm \cite{he1993fixed}.

\begin{thm}[Circle packing theorem]\label{thm:CPT}
Let $ G $ be a fully triangulated graph of genus $ g $. There exist a Riemann surface $ S $ of genus $ g $ and a univalent circle packing $ \mathcal{P} $ such that $ \mathcal{P} $  is a circle packing for $ G $ on $ S $. Moreover, this packing is unique up to the automorphisms of $ S $.
\end{thm}

If $ G $ is embedded in a surface of genus $ g $ but is not fully triangulated, the Riemann surface and the circle packing as guaranteed by the theorem still exist, although they may not be unique \cite{Kelner2006Spec}.

The complex structure on the Riemann surface naturally induces a notion of the angle between two edges of the same face.
For a face with vertices $v, u_i, u_{i+1}$, let $\langle u_i v u_{i+1}\rangle$ denote the angle between the edges $\{v,u_i\}$ and $\{v,u_{i+1}\}$.
Consider a vertex $v$ and all its incident faces.
The angles at $v$ that correspond to these faces form a cyclic sequence: $\langle u_1 v u_2 \rangle, \langle u_2 v u_3 \rangle, \dots, \langle u_\ell v u_1 \rangle$.
The sequence of angles wraps around the vertex  $v$ an integer number of times before returning to $u_1$.
Hence, the total angle $\sum_{i=1}^{\ell} \langle u_i v u_{i+1}\rangle$  equals $2k\pi$ for some integer $k$, where $u_{\ell+1}=u_1$.
When $k\geq 2$, the vertex $v$ is called a \emph{discrete branch point of order $k$}.
If $\mathcal{P}$ is a univalent circle packing on the sphere, the angle sum at every vertex of $A(\mathcal{P})$ is $2\pi$.
For a non-univalent circle packing $\mathcal{P}$, the edges of $A(\mathcal{P})$ may cross, yet $A(\mathcal{P})$ still provides a well-defined triangulation of the surface, although a point on the surface may be covered by more than one triangle (face). In this case, the integer $k$ may exceed $1$.

It is well established that a large portion of Riemann surface theory has discrete counterparts in the study of circle packings. 
Analogous to analytic maps between Riemann surfaces, one can define discrete maps between circle packings that send circles on one surface to circles on another while preserving tangencies.
Analytic maps send infinitesimal circles to infinitesimal circles, whereas circle packing maps send finite circles to finite circles. 
The discrete analogue of branched covering maps in Riemannian geometry sends univalent circle packings on one surface to non-univalent circle packings on another.
However, such discrete maps occur far less frequently than their continuous analogues. 
In particular, given a univalent circle packing $\mathcal{P}$ on a surface $S$ of genus $g$, there does not always exist a non-univalent packing with the same associated graph $A(\mathcal{P})$ on the Riemann sphere $\widehat{\mathbb{C}}$.
The difficulty arises from constraints on branch points. For an analytic map $f: S \to \widehat{\mathbb{C}}$, the branch points are in a highly restricted set. In contrast, for a circle packing map, branch points can occur only at the centers of the circles.
If there is no admissible set of branch points among the centers of the circles, it is very challenging to construct a discrete analytic map from $S$ to $\widehat{\mathbb{C}}$.

To overcome this technical difficulty, we refine the graph $G$ as described in Section \ref{sec:ref}.  
The refinement ensures that the circle packing consists of infinitesimal circles and that the branch points of the map $f$ lie at the centers of some of these circles.

\subsection{Some refinements of graphs}\label{sec:ref}
Let $G$ be a  fully triangulated graph.
Define the \emph{hexagon-subdivision} refinement of $G$ by replacing each triangular face of $G$ with four smaller triangles, as shown in Figure \ref{fig:hex-subdivision}.
This operation produces a new fully triangulated graph, and hence  can be iterated arbitrarily many times.
Denote by $G^{(k)}$ the graph obtained from $G$ after $k$ iterations of the hexagon-subdivision, and by $S^{(k)}$ the Riemann surface on which $G^{(k)}$ is realized a circle packing.
Note that each hexagon-subdivision operation may change the shapes and sizes of the original triangles.
Nevertheless, the sequence of Riemann surfaces $\{S^{(k)}\}$ converges, and the embeddings of the graphs $G^{(k)}$ on these surfaces become stable.
The following convergence result was first established by Bowers and Stephenson \cite{Bowers2004Unifor} and later restated by Kelner \cite{Kelner2006Spec} in the following form.

\begin{lem}\cite[Lemma 5.4]{Kelner2006Spec}\label{lem:con}
Let $G$ be a graph that triangulates a genus $g$ compact Riemann surface without boundary, let $G^{(k)}$ be the result of performing $k$ hexagon-subdivision
on $G$, and let $S^{(k)}$ be the Riemann surface on which $G^{(k)}$ is realized as a circle packing. Further, let $h_k : S^{(k)} \rightarrow S^{(k+1)}$ be the map that takes a triangle to its image under the subdivision procedure by the obvious piecewise-linear map. The sequence of surfaces
$\{S^{(k)}\}$ converges in the moduli space of genus $g$ surfaces, and the sequence of maps $\{h_k\}$ converges to the identity.
\end{lem}

Lemma \ref{lem:con} implies that, in the limit, the circle packing for $G^{(\infty)}$ on $S^{(\infty)}$ consists of circles with arbitrarily small radii.

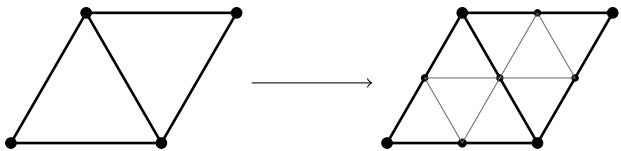
\begin{figure}[H]
    \centering
\begin{tikzpicture}[scale=1]
  \foreach \i in {0,2} {
    \filldraw[black] (\i,0) circle (2pt);
  }
  \foreach \i in {1,3} {
    \filldraw[black] (\i,1.732) circle (2pt);
  }

\draw[line width=1pt] (0,0) -- (2,0) -- (3,1.732) -- (1,1.732) -- cycle;
 \draw[line width=1pt]  (2,0)  -- (1,1.732);

\foreach \i in {5,7} {
    \filldraw[black] (\i,0) circle (2pt);
}

\foreach \i in {6,8} {
    \filldraw[black] (\i,1.732) circle (2pt);
}

\filldraw[black] (6,0) circle (1.5pt);
\filldraw[black] (7,1.732) circle (1.2pt);

\foreach \i in {5.5,6.5,7.5} {
    \filldraw[black] (\i,0.866) circle (1.2pt);
}

\draw[line width=1pt] (5,0) -- (7,0) -- (8,1.732) -- (6,1.732) -- cycle;
\draw[line width=1pt] (7,0) -- (6,1.732);

\draw[line width=0.1pt, black!60]
    (6.5,0.866) -- (5.5,0.866) -- (6,0)
    -- (6.5,0.866) -- (7.5,0.866)
    -- (7,1.732) -- (6.5,0.866);
    
\draw[->] (3.2,0.8) -- (4.8,0.8);
\end{tikzpicture}
    \caption{The hexagon-subdivision refinement of two triangle.}
    \label{fig:hex-subdivision}
\end{figure}

A graph $H$ is said to contain $G$ as a \emph{$(\xi,\ell)$-immersion} if $H$ can be obtained by replacing each edge of $G$ with a path of length at most $\ell$ connecting its endpoints, in such a way that each edge of $H$ is used by at most $\xi$ of these paths.
Note that when $\xi = 1$, the paths are edge-disjoint, thus recovering the usual notion of immersion.
We present a useful comparison result for Steklov eigenvalues for graphs that contain $G$ as a $(\xi,\ell)$-immersion.
\begin{thm}\label{thm:sd}
Let $G$ be a graph with boundary $\delta \Omega$.
If $H$ contains $G$ as a $(\xi,\ell)$-immersion, then
\begin{equation*}
  \lambda_k(G,\delta \Omega) \leq \xi \ell \lambda_k(H, \delta \Omega)
\end{equation*}
 for $k=1,\ldots,|\delta \Omega|$.
\end{thm}

In the following, we sometimes consider a graph $H$ that contains a graph $\widehat{G}$, which is isomorphic to $G$ but has a different vertex set, as a $(\xi,\ell)$-immersion. 
For simplicity, we still say that $H$ contains $G$ as a $(\xi,\ell)$-immersion.  
In this situation, $\delta \Omega$ is no longer a subset of $V(H)$.
However, there exists a subset of $V(H)$ that corresponds to $\delta \Omega$, and for convenience, we will use the same notation $\delta \Omega$ to denote it.

\begin{lem}\label{Lem:H}
Let $G$ be a fully triangulated graph of bounded-degree with boundary $\delta \Omega$.
Then there exist a subgraph $H$ and a boundary set $\delta\Omega^{(k)}$ of $G^{(k)}$ such that $H$ contains $G$ as an $(O(1),O(2^k))$-immersion,  $|\delta\Omega^{(k)}|=\Theta(4^k)|\delta \Omega|$ and $$\lambda_2(H, \delta \Omega) \leq O(2^k) \lambda_2(G^{(k)}, \delta\Omega^{(k)}).$$
\end{lem}

From Theorem \ref{thm:sd} and Lemma \ref{Lem:H}, it follows that there exists a graph $H$ such that  
$$|\delta \Omega| \lambda_2(G, \delta \Omega) \leq O(2^k)|\delta \Omega|\lambda_2(H, \delta \Omega) \leq O(4^k)|\delta \Omega| \lambda_2(G^{(k)}, \delta\Omega^{(k)}).$$
Then we get the following result, which shows that it suffices to prove Theorem \ref{thm:main} for an appropriate hexagonal-subdivision of the original graph $G$.

\begin{lem}\label{lem:main}
Let $G$ be a fully triangulated graph of bounded-degree with boundary $\delta \Omega$.
Then there exist a boundary set $\delta\Omega^{(k)}$ of $G^{(k)}$ and a constant $C$ such that $$|\delta \Omega| \lambda_2(G, \delta \Omega) \leq C|\delta\Omega^{(k)}| \lambda_2(G^{(k)}, \delta\Omega^{(k)}).$$
\end{lem}

\section{Proofs}\label{sec:proof}

In this section, we provide all the proofs appearing in this paper.

\begin{proof}[Proof of Theorem \ref{thm:sd}]
Recall that $H$ contains $G$ as a $(\xi,\ell)$-immersion. It implies that for every edge $\{x,y\} \in E(G)$, there exists a corresponding path $P_{xy}$ in $H$ of length at most $\ell$.
By the Cauchy-Schwarz Inequality, we have
\begin{equation*}
   (f(x)-f(y))^2\leq \ell \sum_{\{u,v\} \in E(P_{xy})}(f(u)-f(v))^2
\end{equation*}
for any real function $f$. 
Since each edge of $H$ is used by at most $\xi$ of these paths, we get
\begin{align*}
   \sum_{\{x,y\} \in E(G)}(f(x)-f(y))^2&\leq \ell \sum_{\{x,y\} \in E(G)}\sum_{\{u,v\} \in E(P_{xy})}(f(u)-f(v))^2\\
   &\leq \ell\xi \sum_{\{x,y\} \in E(H)}(f(x)-f(y))^2.
\end{align*}
It follows from \eqref{eq:Rq1} that
\begin{align*}
  \lambda_k(G,\delta \Omega)
  &= \min\limits_{\substack{W\subset \mathbb{R}^{V}, \dim W = k-1, \\ W\perp \mathbf{1}_{\delta \Omega}}}
     \max\limits_{f \in W} \frac{\sum_{\{x,y\} \in E(G)}(f(x)-f(y))^2}{\sum_{x\in \delta \Omega}f^2(x)}\\
    &\leq \min\limits_{\substack{W\subset \mathbb{R}^{V}, \dim W = k-1, \\ W\perp \mathbf{1}_{\delta \Omega}}}
     \max\limits_{f \in W} \frac{\ell\xi\sum_{\{x,y\} \in E(H)}(f(x)-f(y))^2}{\sum_{x\in \delta \Omega}f^2(x)}\\
    &=\ell\xi\lambda_k(H,\delta \Omega).
\end{align*}
\end{proof}

Lemma \ref{Lem:H} is a key lemma in this paper. Its proof starts by constructing a boundary set $\delta\Omega^{(k)}$ of $G^{(k)}$ such that $|\delta\Omega^{(k)}| = \Theta(4^k) |\delta\Omega|$. 
Next, a random subgraph $H$ of $G^{(k)}$ with boundary $\delta \Omega$ is constructed in a specific manner, ensuring that $H$ contains $G$ as an $(O(1), O(2^k))$-immersion.  
By analyzing the expectation of the Steklov eigenvalues of the random subgraph $H$, we  conclude that there exists a certain $H$ such that $\lambda_2(H, \delta\Omega) \le O(2^k)\, \lambda_2(G^{(k)}, \delta\Omega^{(k)})$.

\begin{proof}[Proof of Lemma \ref{Lem:H}]
Let $V(G)=\{v_1,\dots,v_n\}$.
We define a partition $\{V^{(k)}_1,\dots,V^{(k)}_n\}$ of $V(G^{(k)})$ by assigning each vertex of $G^{(k)}$ to the subset $V^{(k)}_i$ corresponding to its nearest vertex $v_i \in V(G)$, with ties broken arbitrarily. Note that $v_i \in V^{(k)}_i$ for all $i$.
Let $r=2^k$. The construction of $G^{(k)}$ ensures that $|V^{(k)}_i|=\Theta(r^2)$ for each $i$.
The boundary of $G^{(k)}$ is defined as
$$\delta\Omega^{(k)}:=\bigcup_{i:v_i\in \delta \Omega}{V^{(k)}_i},$$
which satisfies $|\delta\Omega^{(k)}|=\Theta(r^2)|\delta \Omega|$.

Next, we construct a random subgraph $H$ of $G^{(k)}$ that contains $G$ as a $(\xi,\ell)$-immersion.
For each $i$, we randomly select a vertex from $V^{(k)}_i$ to serve as a representative of $v_i$, denoted by $\pi_V(v_i)$.
Since $|V^{(k)}_i| = \Theta(r^2)$ for each $i$, each vertex is chosen with probability $1/\Theta(r^2)$.

For each edge $\{v, u\} \in E(G)$, we construct a random path in $G^{(k)}$ between $\pi_V(v)$ and $\pi_V(u)$ as follows.
Let $T_1$ and $T_2$ be the two triangles (faces) in $G$ sharing the edge $\{v, u\}$.
For a triangle $T$ in $G$, let $T^{(k)}$ denote the set of all small triangles in $G^{(k)}$ obtained by applying the hexagon-subdivision to $T$.
We choose a vertex $x$ uniformly at random from $T_1^{(k)} \cup T_2^{(k)}$, then construct random paths from $\pi_V(v)$ to $x$ and from $\pi_V(u)$ to $x$, which we call the two  \emph{initial-paths}.  
The desired path is then obtained by joining the initial-paths at $x$, while removing any duplicate edges that may appear.

With loss of generality, we just describe the construction of the path from $\pi_V(v)$ to $x$.
Let $\mathcal{T}(v)$ be the set of all triangles in $G$ that contain $v$.
Let $T_{\pi}$ and $T_x$ denote the triangles of $G$ such that $\pi_V(v)$ lies in $T_{\pi}^{(k)}$ and $x$ lies in $T_x^{(k)}$, respectively.
Note that both $T_{\pi}$ and $T_x$ belong to $\mathcal{T}(v)$. The triangles in $\mathcal{T}(v)$ have a natural cyclic order around $v$, which defines two sequences between $T_{\pi}$ and $T_x$: one clockwise and one counterclockwise.
We select the shorter sequence, denoted by $T_0, T_1, \ldots, T_{\ell+1}$, breaking ties randomly,  where $T_0 = T_{\pi}$ and $T_{\ell+1} = T_x$.
All triangles in this sequence are distinct unless $T_{\pi}= T_x$.
 
We now select random vertices $a_i \in V(G^{(k)})$ from each $T_i^{(k)}$ for $i=1,\ldots,\ell$, and set $a_0 = \pi_V(v)$ and $a_{\ell+1} = x$.
To connect $a_i$ to $a_{i+1}$: if $T_i \neq T_{i+1}$, we observe that $T_i^{(k)}$ and $ T_{i+1}^{(k)}$ together form an $(r+1) \times (r+1)$ grid, see Figure \ref{fig:grid}. There are two types of paths from $a_i$ to $a_{i+1}$: first moving horizontally along the row direction and then vertically along the column direction, or vice versa. Randomly choose one of these two paths. If two vertices lie in the same row or column, they are connected directly along that row or column; if $T_{i} = T_{i+1}$, we randomly select an adjacent triangle of $T_{i}$ to form the grid and then apply the same path construction. Connecting these segment paths yields the initial-path from $\pi_V(v)$ to $x$, see Figure \ref{fig:initialpath} for example.

Using the initial-paths from $\pi_V(v)$ and $\pi_V(u)$ to $x$, one can naturally derive a random path $\pi_P(v,u)$ connecting $\pi_V(v)$ and $\pi_V(u)$.
Applying this procedure independently to each edge of $G$ yields a subgraph $H$ of $G^{(k)}$, which contains $G$ as $(\xi,\ell)$-immersion. Since $G$ has bounded degree, each initial-path has length $O(r)$, which implies that $\ell = O(r)$. For any triangle $T$ in $G$, each initial-path using edges of $T^{(k)}$ has an endpoint in the set $\{\pi_V(v): v \in V(T)\}$. Note that the number of initial-paths with endpoint $\pi_V(v)$ is at most the degree of $v\in V(G) $. Hence, each edge of $T^{(k)}$ is used by at most $O(1)$ initial-paths, which implies that $\xi = O(1)$.

Recall that each initial-path is constructed within an $O(r) \times O(r)$ grid and uses $O(r)$ edges of the grid.
Hence, the expected number of times that a given edge of $G^{(k)}$ occurs in the random graph $H = (\pi_V, \pi_P)$ is $O\left(\frac{1}{r}\right)$. Then we have 
\begin{equation}\label{eq:edge}
   \mathbb{E}[\sum_{\{x,y\}\in E(H)}(f(x)-f(y))^2]\leq O\left(\frac{1}{r}\right)\sum_{\{x,y\}\in E(G^{(k)})}(f(x)-f(y))^2
\end{equation}
for any real function $f$. Recall that each vertex in $\delta\Omega^{(k)}$ is chosen with probability $\Theta(1/r^2)$, then we have 
\begin{equation}\label{eq:vertex}
   \mathbb{E}[\sum_{\{x,y\}\subset \delta \Omega}(f(x)-f(y))^2]=\Theta\left(\frac{1}{r^4}\right)\sum_{\{x,y\}\subset \delta\Omega^{(k)}}(f(x)-f(y))^2.
\end{equation}
From \eqref{eq:edge} and \eqref{eq:vertex}, 
it follows that there exists some choice of $H=(\pi_V, \pi_P)$ such that 
\begin{equation*}
   \frac{\sum_{\{x,y\}\in E(H)}(f(x)-f(y))^2}{\sum_{\{x,y\}\subset \delta \Omega}(f(x)-f(y))^2}\leq O\left(r^3\right)\frac{\sum_{\{x,y\}\in E(G^{(k)})}(f(x)-f(y))^2}{\sum_{\{x,y\}\subset \delta\Omega^{(k)}}(f(x)-f(y))^2}.
\end{equation*}
By \eqref{eq:Ray}, we obtain that  
\begin{align*}
\lambda_2(H,\delta \Omega)
  &=|\delta \Omega|\min_f\frac{\sum_{\{x,y\} \in E(H)}(f(x)-f(y))^2}{\sum_{\{x,y\}\subset \delta \Omega}(f(x)-f(y))^2}\\
  &\leq \frac{|\delta \Omega|}{|\delta\Omega^{(k)}|}O\left(r^3\right)\lambda_2(G^{(k)}, \delta\Omega^{(k)})\\
  &=O(r)\lambda_2(G^{(k)}, \delta\Omega^{(k)}).
\end{align*}
\end{proof}
\begin{figure}[htbp]
    \centering
     \begin{tikzpicture}[scale=0.40]
  \begin{scope}[xslant=-0.5]
    \foreach \x in {1,...,8} {
      \draw[line width=0.05pt, black!30] (\x,0) -- (\x,9);
    }
    \foreach \y in {1,...,8} {
      \draw[line width=0.05pt, black!30] (0,\y) -- (9,\y);
    }

    \foreach \i in {0,...,8} {
      \foreach \j in {0,...,8} {
        \draw[line width=0.05pt, black!30] (\i,\j) -- (\i+1,\j+1);
      }
    }

    \draw[line width=1.5pt] (0,0) rectangle (9,9);

    \draw[line width=0.6pt] (0,0) -- (9,9);

    \node[below left] at (0,0) {$0$};
    \node[above left] at (0,9) {$r$};
    \node[below right] at (9,0) {$r$};

    \draw[black, line width=1.2pt] (2,6) -- (2,2) -- (7,2);
  \end{scope}
 
  \pgfmathsetmacro{\ax}{2 - 0.5*6}
  \pgfmathsetmacro{\ay}{6}
  \pgfmathsetmacro{\bx}{7 - 0.5*2}
  \pgfmathsetmacro{\by}{2}

  \filldraw[black] (\ax,\ay) circle (4pt) node[above right] {\small{$a_i$}};
  \filldraw[black] (\bx,\by) circle (4pt) node[below] {\small{$a_{i+1}$}};

\end{tikzpicture}
    \caption{The grid graph induced by two adjacent triangles.}
    \label{fig:grid}
\end{figure}
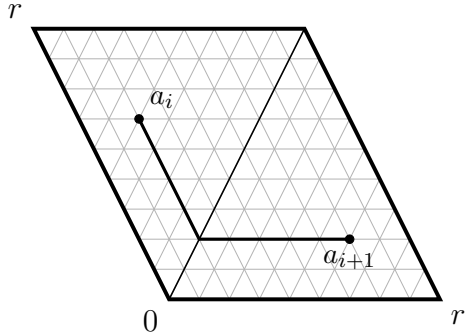
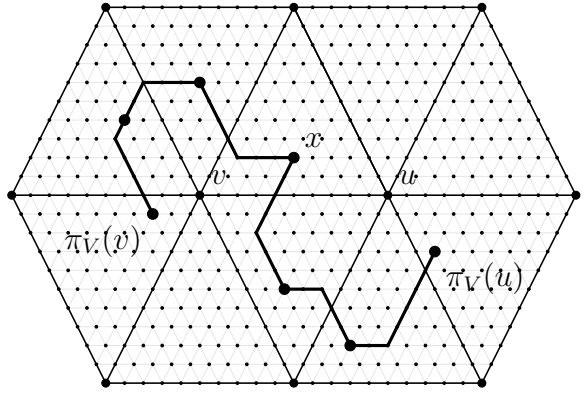
\begin{figure}[htbp]
    \centering
    \begin{tikzpicture}[scale=0.25]
  \begin{scope}[xslant=-0.5]
    \coordinate (A) at (0,0);
    \coordinate (B) at (10,0);
    \coordinate (C) at (10,10);
    \coordinate (D) at (0,10);
\coordinate (E) at (20,10);
\coordinate (F) at (20,20);
\coordinate (G) at (10,20);
\coordinate (H) at (30,20);
\coordinate (I) at (30,10);
\coordinate (J) at (20,0);

\draw[line width=0.6pt] (A) -- (B) -- (C) -- (D) -- cycle;
\draw[line width=0.6pt] (A) --(C) -- (E) -- (B);
\draw[line width=0.6pt] (C) -- (F) -- (E);
\draw[line width=0.6pt] (C) -- (G) -- (F);
\draw[line width=0.6pt] (D) -- (G);
\draw[line width=0.6pt] (E) -- (F) -- (H);
\draw[line width=0.6pt] (E) -- (H) -- (I);
\draw[line width=0.6pt] (E) -- (I) -- (J);
 \draw[line width=0.61pt] (E) -- (J) -- (B);

\foreach \x in {1,...,9} {
  \draw[line width=0.1pt, gray!20] (\x,0) -- (\x,10+\x);
    \draw[line width=0.1pt, gray!20] (10+\x,0) -- (10+\x,20);
      \draw[line width=0.1pt, gray!20] (20+\x,\x) -- (20+\x,20);
}

\foreach \y in {1,...,9} {
  \draw[line width=0.1pt, gray!20] (0,\y) -- (20-\y,20);
  \draw[line width=0.1pt, gray!20] (\y,0) -- (20+\y,20);
    \draw[line width=0.1pt, gray!20] (10+\y,0) -- (30,20-\y);
}

\foreach \z in {1,...,9} {
	\draw[line width=0.1pt, gray!20] (0,\z) -- (20+\z,\z);
		\draw[line width=0.1pt, gray!20] (\z,10+\z) -- (30,10+\z);
}

    \draw[line width=1.2pt] (7,9) -- (7,13) -- (8,14) -- (10,16) -- (13,16) --(13,12) -- (16,12) -- (12,8) -- (12,5) -- (14,5) -- (14,2) -- (16,2) -- (21,7);
    
  \end{scope}
  
    \foreach \i in {0,...,10} {
      \pgfmathtruncatemacro{\jmax}{10+\i}
    \foreach \j in {0,...,\jmax} {
    \pgfmathsetmacro{\ii}{\i - 0.5*\j}
    \pgfmathsetmacro{\jj}{\j}
    \filldraw[black] (\ii,\jj) circle (2pt);
  }
  }
 
    \foreach \i in {11,...,20} {
	\foreach \j in {0,...,20} {
		\pgfmathsetmacro{\ii}{\i - 0.5*\j}
		\pgfmathsetmacro{\jj}{\j} 		\filldraw[black] (\ii,\jj) circle (2pt);
	}
}
  
      \foreach \i in {21,...,30} {
  	      \pgfmathtruncatemacro{\jmin}{\i-20}
  	    \foreach \j in {\jmin,...,20} {
  		    \pgfmathsetmacro{\ii}{\i - 0.5*\j}
  		    \pgfmathsetmacro{\jj}{\j}
  		    \filldraw[black] (\ii,\jj) circle (2pt);
  		  }
  	  }
  
    \foreach \x/\y/\name in {
    0/0/A, 10/0/B,0/10/D, 20/20/F, 10/20/G, 30/20/H,
    30/10/I, 20/0/J
  }{
    \pgfmathsetmacro{\xx}{\x - 0.5*\y}
    \pgfmathsetmacro{\yy}{\y}
    \filldraw[black] (\xx,\yy) circle (6pt) ;
  }
    \pgfmathsetmacro{\cx}{10- 0.5*10}
    \pgfmathsetmacro{\cy}{10}
    \filldraw[black] (\cx,\cy) circle (6pt) node[above right] {$v$};
      \pgfmathsetmacro{\ex}{20- 0.5*10}
    \pgfmathsetmacro{\ey}{10}
    \filldraw[black] (\ex,\ey) circle (6pt) node[above right] {$u$};
 
\pgfmathsetmacro{\xx}{7- 0.5*9}
\pgfmathsetmacro{\xy}{9}
  \filldraw[black] (\xx,\xy) circle (8pt) node[below left] {$\pi_V(v)$};

\pgfmathsetmacro{\ax}{8- 0.5*14}
\pgfmathsetmacro{\ay}{14}
\filldraw[black] (\ax,\ay) circle (8pt);

\pgfmathsetmacro{\ax}{13- 0.5*16}
\pgfmathsetmacro{\ay}{16}
\filldraw[black] (\ax,\ay) circle (8pt);

     \pgfmathsetmacro{\xx}{16- 0.5*12}
    \pgfmathsetmacro{\xy}{12}
    \filldraw[black] (\xx,\xy) circle (8pt) node[above right] {$x$};
 
 \pgfmathsetmacro{\xx}{21- 0.5*7}
 \pgfmathsetmacro{\xy}{7}
 \filldraw[black] (\xx,\xy) circle (8pt) node[below right] {$\pi_V(u)$};

    \pgfmathsetmacro{\ax}{12- 0.5*5}
    \pgfmathsetmacro{\ay}{5}
    \filldraw[black] (\ax,\ay) circle (8pt);
    
    \pgfmathsetmacro{\ax}{14- 0.5*2}
    \pgfmathsetmacro{\ay}{2}
    \filldraw[black] (\ax,\ay) circle (8pt);
    
\end{tikzpicture}
    \caption{Two initial-paths between $\pi_V(v)$ and $\pi_V(u)$.}
    \label{fig:initialpath}
\end{figure}

The following result is a discrete version of the classical \emph{center of mass method}, which can be directly derived from the work of Spielman-Teng \cite[Theorem 4.2]{ST2007}.


\begin{lem}\label{lem:centroid}
    Let $f: S\to \mathbb{S} $ be an analytic map from  a Riemann surface $S$ to the unit sphere $\mathbb{S}$ in $\mathbb{R}^3$. For $i=1,\ldots ,n$, let $C_i$ be a circle with center $c_i$ on $S$. 
    Suppose that no point on $\mathbb{S}$ lies in at least half of the regions that have boundaries $f(C_1),\ldots,f(C_n)$.
    Then there is a M\"{o}bius transformation $\Psi$ such that 
$$\frac{\sum_{i=1}^{n}\Psi(f(c_i))}{n}=\mathbf{0}$$
where $\mathbf{0}$ is the zero vector in $\mathbb{R}^3$.
\end{lem}

If the graph $G$ is not fully triangulated, additional edges can be added to make it.
By \eqref{eq:orig}, this operation does not decrease $\lambda_2(G,\delta \Omega)$.
The following observation shows that there exists a way to add edges while keeping the graph with bounded-degree and genus $g$.
Therefore, it suffices to prove Theorem \ref{thm:main} for a fully triangulated graph $G$.

\begin{Obs}\label{obs:full}
For a bounded-degree graph $G$ of genus $g$, there exists a fully triangulated graph $H$ of genus $g$ with bounded degree such that $G$ is a spanning subgraph of $H$.
\end{Obs} 
\begin{proof}
If a face $F$ of $G$ has a boundary that is not a cycle, we can add at most two edges to each vertex inside $F$ so that its boundary vertices form a cycle. This process subdivides $F$ into new faces, each bounded by a cycle, while preserving the genus of $G$.
Then, we may assume that the boundary of every face of $G$ is a cycle.

Consider a face of $G$ whose boundary consists of vertices $v_0, v_1, \ldots,v_s,v_0$ that appear in cyclic order.
If the face is not a triangle, we add a path $v_1v_sv_2v_{s-1}\cdots v_{t}$, where $t=\frac{s}{2}$ for even $s$ and $t=\lceil\frac{s+2}{2}\rceil$ for odd $s$ to make it fully triangulated. Since the number of faces incident with a vertex is at most its degree, the increase of the degree of each vertex is bounded. Thus, we get a fully triangulated graph of genus $g$ with bounded degree that contains $G$ as a spanning subgraph.
\end{proof}

We are now ready to present the proof of the main result in this paper.

\begin{proof}[Proof of Theorem \ref{thm:main}]
Let $G$ be a graph with boundary $\delta \Omega$. 
By \eqref{eq:orig}, adding edges to $G$ does not decrease $\lambda_2(G, \delta \Omega)$. 
If $G$ is not fully triangulated, Observation \ref{obs:full} shows that additional edges can be inserted into $G$ to obtain a fully triangulated bounded-degree graph of genus $g$. 
Therefore, it suffices to consider the case where $G$ is fully triangulated.

Recall that $G^{(k)}$ denotes the graph obtained from $G$ after $k$ iterations of the hexagonal-subdivision. 
Let $\delta \Omega^{(k)}$ be the boundary of $G^{(k)}$ as defined in Lemma \ref{lem:main}.
By circle packing theorem (Theorem \ref{thm:CPT}), there exists a Riemann surface $S^{(k)}$ of genus $g$ such that $G^{(k)}$ realize a univalent circle packing $\mathcal{P}^{(k)}$ on $S^{(k)}$.
By Theorem~\ref{thm:RR}, one can find an analytic map $f^{(k)}: S^{(k)} \to \widehat{\mathbb{C}}$ of degree $O(g)$ with $O(g)$ branch points.
Embedding the Riemann sphere $\widehat{\mathbb{C}}$ as the unit sphere $\mathbb{S}$ in $\mathbb{R}^3$ via the inverse stereographic projection, we obtain a composition $ \mathbf{v}^{(k)}: S^{(k)} \xrightarrow{f^{(k)}} \widehat{\mathbb{C}} \hookrightarrow \mathbb{S},$
which defines an analytic map from $S^{(k)}$ to $\mathbb{S}$.

For each vertex $v \in G^{(k)}$, let $C^{(k)}_v$ denote the circle in $\mathcal{P}^{(k)}$ centered at $v$. 
The circle $C^{(k)}_v$ is mapped by $\mathbf{v}^{(k)}$ to a connected region bounded by $\mathbf{v}^{(k)}(C^{(k)}_v)$ on $\mathbb{S}$. Moreover, each point on $\mathbb{S}$ is covered by at most $O(g)$ such regions.
Since $|\delta\Omega^{(k)}|=|\delta\Omega|\Theta(4^k)$, when $k$ is sufficiently large with respect to $g$, no point on $\mathbb{S}$ lies in more than half of the regions bounded by $\mathbf{v}^{(k)}(C^{(k)}_v)$ for all $v \in \delta\Omega^{(k)} $.
By Lemma \ref{lem:centroid}, we may assume that the boundary vertices satisfy $\sum_{v\in \delta \Omega^{(k)}} \mathbf{v}^{(k)}(v)=\mathbf{0}$; 
otherwise, $\mathbf{v}^{(k)}$ can be post-composed with a suitable M\"{o}bius transformation to achieve this normalization. 
Since all of the $\mathbf{v}^{(k)}(v)$ are on the unit sphere $\mathbb{S}$, we have ${\sum_{v\in \delta \Omega^{(k)}}\lVert \mathbf{v}^{(k)}(v)\rVert^2}=| \delta \Omega^{(k)}|$.

By the result of Lin-Zhao \cite[Lemma 2.4]{lin2025steklov}, the scalar values $f(x)$ in \eqref{eq:orig} can be replaced with the vectors $\mathbf{v}(x) \in \mathbb{R}^n$ in \eqref{eq:orig}, that is, 
\begin{align*}
\lambda_2(G,\delta \Omega)&=\min\frac{\sum_{\{x,y\} \in E}\lVert \mathbf{v}(x)-\mathbf{v}(y) \rVert^2}{\sum_{x\in \delta \Omega}\lVert \mathbf{v}(x)\rVert^2},
\end{align*}
where the minimum is taken over all sets of $n$-dimensional vectors such that $\sum_{x\in \delta \Omega} \mathbf{v}(x)=\mathbf{0}$ and at least one $\mathbf{v}(x)$ is nonzero. Thus, we have 
\begin{align}
    \lambda_2(G^{(k)},\delta \Omega^{(k)}) &\leq\frac{\sum_{\{u,v\} \in E(G^{(k)})}\lVert \mathbf{v}^{(k)}(u)-\mathbf{v}^{(k)}(v) \rVert^2}{\sum_{v\in \delta \Omega^{(k)}}\lVert \mathbf{v}^{(k)}(v)\rVert^2} \notag\\
    & =\frac{\sum_{\{u,v\} \in E(G^{(k)})}\lVert \mathbf{v}^{(k)}(u)-\mathbf{v}^{(k)}(v) \rVert^2}{| \delta \Omega^{(k)}|}.\label{eq:num}
\end{align}
Then it remains to bound the numerator.

Away from its branch points, the analytic map \(\mathbf{v}^{(k)}: S^{(k)} \to \mathbb{S}\) is locally conformal,
which implies that infinitesimal circles on the compact Riemann surface $S^{(k)}$ are mapped to infinitesimal circles on $\mathbb{S}$.
More formally, for any $\epsilon, \kappa >0$, there exists $\tau>0$ such that for any $\tau'<\tau$ and any point $p$ at least distance $\kappa$ from all branch points of $\mathbf{v}^{(k)}$,  the radius $\tau'$ distortion
\begin{align}\label{eq:dis}
  \frac{\max_{q:|p-q|=\tau'}|\mathbf{v}^{(k)}(p)-\mathbf{v}^{(k)}(q)|}{\min_{q:|p-q|=\tau'}|\mathbf{v}^{(k)}(p)-\mathbf{v}^{(k)}(q)|}-1 < \epsilon.
\end{align}
Indeed, by the convergence result of $S^{(k)}$ (Lemma \ref{lem:con}), 
there exists some $N$ such that for all $k >N$, all of circles in ${\mathcal{P}}^{(k)}$ have radius at most $\tau$. 
Fix $\epsilon$ and $\kappa$, and let $N$ and $\tau$ be chosen such that inequality \eqref{eq:dis} is satisfied for $\mathbf{v}^{(k)}$ with $k >N$. 

For $k > N$ and $i=1,2,3$, let $\mathbf{v}_i^{(k)}$ be the $i$-th coordinate of $\mathbf{v}^{(k)}$. For any $(u,v) \in E(G^{(k)})$, we have 
\begin{align}\label{eq:term}
\lVert \mathbf{v}^{(k)}(u)-\mathbf{v}^{(k)}(v) \rVert^2=\sum_{i=1}^3(\mathbf{v}_i^{(k)}(u)-\mathbf{v}_i^{(k)}(v))^2.
\end{align}
The distance between of $u$ and $v$ on $S^{(k)}$ equals to $r^{(k)}_v+r^{(k)}_u$, where $r^{(k)}_v$ and $r^{(k)}_u$ are the radii of $C^{(k)}_v$ and $C^{(k)}_v$, respectively.
As $k \to \infty$, both $r^{(k)}_v$ and $r^{(k)}_u$ tend to zero. 
The smoothness of the function family $\mathbf{v}^{(k)}$, their convergence to $\mathbf{v}^{(\infty)}$, and the compactness of $S^{(k)}$ ensure that each term on the right-hand side of \eqref{eq:term} is arbitrarily close to its first-order approximation. Thus we have 
\begin{align*}
(\mathbf{v}_i^{(k)}(u)-\mathbf{v}_i^{(k)}(v))^2 &\leq (1+o(1))(r^{(k)}_v+r^{(k)}_u)^2\lVert\nabla \mathbf{v}_i^{(k)}(v)\rVert^2\\
&\leq (2+o(1)) \left((r^{(k)}_v)^2+(r^{(k)}_u)^2 \right)\lVert\nabla \mathbf{v}_i^{(k)}(v)\rVert^2
\end{align*}
and 
\begin{align*}
(\mathbf{v}_i^{(k)}(u)-\mathbf{v}_i^{(k)}(v))^2 \leq (2+o(1)) \left((r^{(k)}_v)^2+(r^{(k)}_u)^2 \right)\lVert\nabla \mathbf{v}_i^{(k)}(u)\rVert^2
\end{align*}
Hence, we have 
\begin{align}\label{eq:app}
(\mathbf{v}_i^{(k)}(u)-\mathbf{v}_i^{(k)}(v))^2 \leq O(1) \left((r^{(k)}_v)^2\lVert\nabla \mathbf{v}_i^{(k)}(v)\rVert^2+(r^{(k)}_u)^2\lVert\nabla \mathbf{v}_i^{(k)}(u)\rVert^2\right).
\end{align}
Since the degree of $G^{(k)}$ is bounded, every vertex of $G^{(k)}$ appears in a constant number of edges. It follows that 
\begin{equation*}
    \sum_{\{u,v\} \in E(G^{(k)})}\lVert \mathbf{v}^{(k)}(u)-\mathbf{v}^{(k)}(v) \rVert^2\leq O(1)\sum_{v\in G^{(k)}}(r^{(k)}_v)^2\lVert\nabla \mathbf{v}^{(k)}(v)\rVert^2.
\end{equation*}
It implies that the contribution per unit area of the circle in $\mathcal{P}^{(k)}$ is at most $O(1)\lVert\nabla \mathbf{v}^{(k)}(v)\rVert^2$, which is bounded as $k$ goes to infinity and $\mathbf{v}^{(k)}$ approaches $\mathbf{v}^{(\infty)}$.

Let $S^{(k)}_1$ be the union of $\kappa$-neighborhoods of the branch points of $\mathbf{v}^{(k)}$ on $S^{(k)}$, and set $S^{(k)}_2 = S^{(k)} \setminus S^{(k)}_1$. 
Consider the circles in $\mathcal{P}^{(k)}$ intersecting $S^{(k)}_1$. 
As $\kappa \to 0$ and $k \to \infty$, the total area of these circles tends to zero, 
since the area of $S^{(k)}_1$ shrinks with $\kappa$ and the circles in $\mathcal{P}^{(k)}$ become arbitrarily small as $k$ increases.
Therefore, their contribution to the numerator in \eqref{eq:num} vanishes in the limit.
Then it remains to consider only those circles that are entirely contained in $S^{(k)}_2$.

For $C \in \mathcal{P}^{(k)}$, let $D(\mathbf{v}^{(k)}(C))$ be the length of the longest geodesic contained in the region bounded by $\mathbf{v}^{(k)}(C)$, and let  $A(\mathbf{v}^{(k)}(C))$  denote its area. When $C$ is entirely contained in $S^{(k)}_2$, the curve $\mathbf{v}^{(k)}(C)$ is approximately circular. 
Suppose that the radius of each circle in $\mathcal{P}^{(k)}$ is at most $\tau$.
In particular, the radius $\tau'$ ($\tau' < \tau$) distortion of $\mathbf{v}^{(k)}$ inside of $S_2^{(k)}$ is at most $\epsilon$ by \eqref{eq:dis}, which implies that $D^2(\mathbf{v}^{(k)}(C))/A(\mathbf{v}^{(k)}(C)) \leq O(1+\epsilon)$.
Let $u,v$ be vertices such that $\{v,u\} \in E(G^{(k)})$ and  both $C_v^{(k)}$ and $C_u^{(k)}$ are entirely contained in $S^{(k)}_2$.
Then we have  that 
\begin{align*}
\lVert \mathbf{v}^{(k)}(u)-\mathbf{v}^{(k)}(v) \rVert^2 &\leq \left(D(\mathbf{v}^{(k)}(C_u^{(k)}))+D(\mathbf{v}^{(k)}(C_v^{(k)}))\right)^2 \\
& \leq 2\left((D^2(\mathbf{v}^{(k)}(C_u^{(k)}))+D^2(\mathbf{v}^{(k)}(C_v^{(k)}))\right) \\
& \leq O(1+\epsilon)\left(A(\mathbf{v}^{(k)}(C_u^{(k)}))+A(\mathbf{v}^{(k)}(C_v^{(k)}))\right).
\end{align*}
Recall that every vertex of $G^{(k)}$ appears in a constant number of edges, and that each point on $\mathbb{S}$ is covered by at most $O(g)$ such regions. 
Consequently,
\begin{equation*}
    \sum_{\{u,v\} \in E(G^{(k)})}\lVert \mathbf{v}^{(k)}(u)-\mathbf{v}^{(k)}(v) \rVert^2\leq O(g(1+\epsilon))A(\mathbb{S}).
\end{equation*}
By \eqref{eq:num}, we have $ \lambda_2(G^{(k)},\delta \Omega^{(k)}) \leq O\left(\frac{g}{|\delta \Omega^{(k)}|}\right)$. 
It follows  from Lemma \ref{lem:main} that  $ \lambda_2(G,\delta \Omega) \leq O\left(\frac{g}{|\delta \Omega|}\right)$.
\end{proof}

\section{Concluding remark}
In this section, we present an application of our results to electrical networks.
If every edge of $G$ is equipped with a resistor of unit resistance, we call $G$ a \emph{unit electrical network}.
A fundamental quantity in such networks is the effective resistance between two vertices, which is determined by the first non-trivial Steklov eigenvalue.

\begin{pro}
Let $G=(V,E)$ be a unit electrical network. For any $u,v\in V$, the effective resistance between $u$ and $v$ is $\frac{2}{\lambda_{2}(G,\{u,v\})}$.
\end{pro}

In practical electrical networks, every component has a bounded number of pins. 
When the layout becomes too complex to be realized on a single-layer printed circuit board (PCB),  engineers may often turn to multi-layer PCBs and use vias to avoid wire crossings.
From a topological perspective, these operations correspond to embedding the network into a surface of higher genus.
Consequently, the graph abstracted from the electrical network  is naturally of bounded degree and positive genus. 
Our result provides a lower bound on the effective resistance between any two vertices in terms of genus of the network.

\begin{thm}
Let $G$ be a unit electrical network of bounded degree and genus $g$. 
Then the effective resistance between  two vertices is at least $c/(g+1)$  for some positive constant $c$.
\end{thm}

\section*{Acknowledgments}
Lixiang Chen was partially supported by the China Postdoctoral Science Foundation (No. 2024M761510) and the National Natural Science Foundation of China (No. 12501485). Yongtang Shi was partially supported by the National Natural Science Foundation of China (No. 12431013).
\bibliographystyle{plain}
\bibliography{ref}

\end{document}